\documentclass[12pt,oneside]{article}
\usepackage{afterpage}
\usepackage{fancyhdr}
\usepackage{lipsum}
\newcommand\shorttitle{Rigidity of Spherically Symmetric Finsler Metrics}
\newcommand\authors{G. Shanker and S. Rani}

\usepackage{supertabular, setspace,hyperref}
\usepackage[a4paper, total={6in, 9in}]{geometry}
\usepackage[utf8]{inputenc}
\usepackage{amsmath}
\usepackage{mathtools}
\usepackage{amsfonts}
\usepackage{amssymb}
\usepackage{amsthm}
\newtheorem{theorem}{Theorem}[section]

\newtheorem{prop}[theorem]{Proposition}
\newtheorem{example}[theorem]{Example}
\newtheorem{cor}[theorem]{Corollary}

\newtheorem{remark}{\sc Remark}
\newtheorem{lemma}{\sc Lemma}[section]
\newtheorem{corollary}{\sc Corollary}[section]
\newtheorem{definition}{\sc Definition}[section]

\newcommand{\be}{\begin{eqnarray}}
\newcommand{\ee}{\end{eqnarray}}
\newcommand{\Be}{\begin{eqnarray*}}
	\newcommand{\Ee}{\end{eqnarray*}}
\newcommand{\bee}{\begin{equation}}
\newcommand{\eee}{\end{equation}}
\newcommand{\ba}{\begin{array}}
	\newcommand{\ea}{\end{array}}
\newcommand{\bl}{\begin{lemma}}
	\newcommand{\el}{\end{lemma}}
\newcommand{\bd}{\begin{definition}}
	\newcommand{\ed}{\end{definition}}
\newcommand{\bt}{\begin{theorem}}
	\newcommand{\et}{\end{theorem}}
\newcommand{\bp}{\begin{proof}}
	\newcommand{\ep}{\end{proof}}
\newcommand{\bi}{\begin{itemize}}
	\newcommand{\ei}{\end{itemize}}
\newcommand{\br}{\begin{remark}}
	\newcommand{\er}{\end{remark}}
\newcommand{\bc}{\begin{corollary}}
	\newcommand{\ec}{\end{corollary}}
\newcommand{\bex}{\begin{example}}
	\newcommand{\eex}{\end{example}}

\usepackage{chngcntr}

\counterwithin*{equation}{section}
\counterwithin*{equation}{section}
\begin{document}
	\afterpage{\cfoot{\thepage}}
	\clearpage
	\date{}
	\title{\textbf{On the Rigidity of Spherically Symmetric Finsler Metrics with Isotropic $E$-Curvature}}
	\maketitle
\begin{center}
		\author{\textbf{Gauree Shanker, Sarita Rani\footnote{Corresponding author}}}
\end{center}
	\begin{center}
		Department of Mathematics and Statistics\\
		School of Basic and Applied Sciences\\
		Central University of Punjab, Bathinda, Punjab-151 001, India\\
		Email:  grshnkr2007@gmail.com, saritas.ss92@gmail.com
\end{center}
\begin{center}
	\textbf{Abstract}
\end{center}
\begin{small}
	In the current  paper, first we give the correct version of the formula for  mean  Berwald curvature of a spherically symmetric Finsler metric given in paper \cite{YCheWSon2015}. Further, we  establish  differential equations characterizing projectively as well as dually flat spherically symmetric Finsler metrics. Finally, we obtain  a rigidity result on spherically symmetric Finsler metrics with isotropic $E$-Curvature.
\end{small}

\section{Introduction}
Finsler spaces, a generalization of Riemannian spaces were named after a German mathematician Paul Finsler who studied the spaces with metric defined by the positive fourth root of a fourth order differential form. Every Riemannian space is Finsler space but not vice-versa. In Riemann-Finsler geometry, it is natural to ask under what conditions a Finsler manifold becomes Riemannian one. This is main problem of what is called Finsler rigidity theory. The best well known result toward this problem are  given by Akbar-Zadeh \cite{HAZad1988} and Numata \cite{SNum1975}. Many authors (\cite{CWKimJWJim2000}, \cite{CWKimJWJim2003}, \cite{ZShe2005}, \cite{BYWu2010}) have worked on this problem with certain assumptions.\\

At the end of $20^{th}$ century, McCarthy and  Rutz \cite{PJMCSFR1993} constructed spherically symmetric Finsler metric of $ \mathbb{R}^4 $ which is invariant under the transformation group of $SO(3)$ rotations. Further, in 2010, Zhou \cite{LZhou2010Arxv} introduced the general spherically symmetric Finsler metric on $ \mathbb{R}^n $ which is invariant under the rotations of  $ \mathbb{R}^n$. He proved that a spherically symmetric Finsler metric $ F(x,y) $ in $ \mathbb{R}^n $ must be of the form $ F= \phi( \arrowvert x \arrowvert, \arrowvert y\arrowvert, \left\langle x, y \right\rangle ), $ where $ \phi $ is a positive smooth function, $ x \in \mathbb{R}^n, \ y \in T_x \mathbb{R}^n, \ \  \left\langle \ , \ \right\rangle $ is the standard inner product on $ \mathbb{R}^n.$\\

Guo et al. \cite{EGHLXM2013} proved that any  spherically symmetric Finsler metric of isotropic Berwald curvature must be a Randers metric. Further, in  \cite{BNajafi2015Arxv}, Najafi classified  locally dually flat and locally projectively flat spherically symmetric Finsler metrics.\\

Recall \cite{ChernShenRFG} that a Finsler metric $F$ on an open subset $\mathcal{U}$ of $\mathbb{R}^n$ is called projectively flat if and only if all the geodesics are straight in $\mathcal{U}$. The real starting point of the investigations of projectively flat metrics is Hilbert's fourth problem \cite{Hilbert4prob}. During the International Congress of Mathematicians,  held in Paris (1900), Hilbert asked about the spaces in which the shortest curves between any pair of points are straight lines. The first answer was given by Hilbert's student G. Hamel in 1903. In \cite{GHAM1903}, Hamel found necessary and sufficient conditions for a  space satisfying a  system of axioms, which is a modification of Hilbert’s system of axioms for Euclidean geometry, to be projectively flat. After Hamel, many authors (see  \cite{SAB.GS.2016PFDF}, \cite{GS.SAB.2015PF}, \cite{GSK2017}, \cite{GS.RKS.RDSK.2017PFDF}, \cite{RY.GS.2013PF}) have worked on this topic. \\

The concept of dually flatness in Riemannian geometry was given by Amari and Nagaoka in \cite{S.I.Ama H.Nag} while they were studying information geometry. Information geometry provides mathematical science with a new framework for analysis. Information geometry is an investigation of differential geometric structure in probability distribution. It is also applicable in statistical physics, statistical inferences etc. Z. Shen \cite{ShenRFGAIG} extended the notion of dually flatness in Finsler spaces.  After Shen's work, many authors  have worked on this topic (see  \cite{SAB.GS.2016PFDF}, \cite{GS.SAB.2017PFDF}, \cite{GSK2017},  \cite{GS.RKS.RDSK.2017PFDF}). Recently, in \cite{GSK2017}, we have worked on projectively flat and dually flat Finsler metrics. Here, we extend our study  to spherically symmetric Finsler spaces. The main aim of this paper is to prove the following rigidity result:

\begin{theorem}{\label{Main_Result}}
	Let $F=u\psi(r,s)$ be a projectively flat and dually flat spherically symmetric Finsler metric of isotropic $E$-curvature, where $ r = \arrowvert x \arrowvert,\ u= \arrowvert y\arrowvert,\ v= \left\langle x, y \right\rangle,\ s=\dfrac{v}{u}$. Then $F$ must be a Riemannian metric. 
\end{theorem}

Current paper is organized as follows:\\
In section 2,  we give basic definitions and results to be used in the study of projectively flat and dually flat spherically symmetric Finsler metrics of isotropic $E$-curvature. In section 3, we provide a formula for  mean Berwald curvature of spherically symmetric Finsler metrics which is a modified version of the formula  given in \cite{YCheWSon2015}.
Further, in section 4, we establish two differential equations completely describing projectively as well as  dually flat  spherically symmetric Finsler metrics.   Finally, we prove our main  rigidity result (i.e. theorem \ref{Main_Result}).

\section{Preliminaries}
In this section, we discuss some basic definitions and results required to study afore said spaces.  We refer  \cite{BCS} and  \cite{ChernShenRFG} for notations and further details.
\begin{definition}
	An n-dimensional real vector space $V$ is called a \textbf{Minkowski space}
	if there exists a real valued function $F:V \longrightarrow [0,\infty)$, called Minkowski norm, satisfying the following conditions: 
	\begin{enumerate}
		\item[(a)]  $F$ is smooth on $V \backslash \{0\},$ 
		\item[(b)] $F$ is positively homogeneous, i.e., $ F(\lambda v)= \lambda F(v), \ \ \forall \ \lambda > 0, $
		\item[(c)] For a basis $\{v_1,\ v_2, \,..., \ v_n\}$ of $V$ and $y= y^iv_i \in V$, the Hessian matrix \\
		$\left( g_{_{ij}}\right)= \left( \dfrac{1}{2} F^2_{y^i y^j} \right)  $ is positive-definite at every point of $V \backslash \{0\}.$
	\end{enumerate} 
\end{definition}


\begin{definition}	
	A connected smooth manifold $M$ is called a Finsler space if there exists a function $F: TM \longrightarrow [0, \infty)$ such that $F$ is smooth on the slit tangent bundle $TM \backslash \{0\}$ and the restriction of $F$ to any $T_pM, \ p \in M$, is a Minkowski norm. Here, $F$ is called a Finsler metric. 
\end{definition}
The geodesic spray coefficients of a Finsler metric $F$ are given by 
$$ G^i \ := \dfrac{1}{4} g^{il} \left\{ \left( F^2\right)_{x^k y^l} y^k - \left( F^2\right)_{x^l}  \right\}, \ \ \text{where}\ \   g^{ij}=\left( g_{ij} \right)^{-1}.$$

Next, we discuss spherically symmetric Finsler metrics in brief. 

\begin{definition}
	A spherically symmetric Finsler metric on $ \mathbb{R}^n$ is a metric that is invariant under all the rotations of $ \mathbb{R}^n.$
\end{definition}
Recall \cite{LZhou2010Arxv} the following result:
\begin{theorem}
	A spherically symmetric Finsler metric $ F(x,y) $ in $ \mathbb{R}^n $ is  of the form $ F= \phi( \arrowvert x \arrowvert, \arrowvert y\arrowvert, \left\langle x, y \right\rangle ), $ where $ x \in \mathbb{R}^n, \ y \in T_x \mathbb{R}^n, \ \phi $ is a positive smooth function, and $  \  \left\langle \ , \ \right\rangle $ is the standard inner product on $ \mathbb{R}^n.$
\end{theorem}
Write $ F=\phi(r, u, v), $ where  $ r = \arrowvert x \arrowvert,\ u= \arrowvert y\arrowvert,\ v= \left\langle x, y \right\rangle.$\\
Since the Finsler metric $F$ satisfies 
$$ F(\lambda\ y)= \lambda\ F(y)\ \ \  \forall\ \  \lambda >0, $$
$\phi$ is positively homogeneous in $u$ and $v$ of degree one, i.e.,
$$ \phi= u\ \phi_u+ v\ \phi_v  \ \ \text{and }\ \  F= \arrowvert y\arrowvert\ \psi(r,s),\ \   \text{where}\ \  s=\dfrac{v}{u}.$$
Recall \cite{XMLZ2014Arxv}  that the geodesic spray coefficients of a spherically symmetric Finsler metric  $F$ are given by
\begin{equation}{\label{geocoeff}}
G^i=uP y^i+u^2Q x^i,
\end{equation}
where 
\begin{equation}{\label{valQ}}
Q=\dfrac{r \psi_{ss}+s \psi_{rs}-\psi_r}{2r(\psi-s \psi_s+(r^2-s^2) \psi_{ss})}
\end{equation}
and 
\begin{equation}{\label{valP}}
P=\dfrac{s \psi_r+r \psi_s}{2r \psi}-\dfrac{s \psi+(r^2-s^2) \psi_{s}}{\psi}\ Q.
\end{equation}
\begin{definition}
	The Berwald curvature of a Finsler metric $F$ is defined, in local co-ordinates, as follows: 
	$$ B:= B^i_{jkl}\ dx^j \otimes dx^k \otimes dx^j=l \otimes \dfrac{\partial}{\partial x^i},$$
	where $B^i_{jkl}=\dfrac{\partial^3 G^i}{\partial y^j \partial y^k \partial y^l}$ and $G^i$ are the geodesic spray coefficients. \\
	A Finsler metric $F$ is called a Berwald metric if its Berwald curvature is zero. 
\end{definition}
Next, we state following lemma for later use:
\begin{lemma}{\label{Lemma}}\cite{XMLZ2014Arxv}
	A spherically symmetric Finsler metric  $F=u \psi(r,s)$ is a Berwald metric if and only if $P$ and $Q$ in its geodesic spray coefficients satisfy  the following   equations:
	\begin{equation}
	\begin{split}
	sP_s-P&=0,\\
	P_{ss}&=0,\\
	sQ_{ss}-Q_s&=0,\\
	Q_{sss}&=0.
	\end{split}
	\end{equation}
\end{lemma}


In the study of spherically symmetric Finsler metrics, concepts of locally projectively flat and locally dually flat are quite important. Here, we discuss some related basic definitions and results.\\
\begin{definition}
	Two Finsler metrics $F$ and $\bar{F}$ on a manifold $M$ are called projectively equivalent if they have same geodesics as point sets, i.e., for any geodesic $\bar{\sigma}(\bar{t}) $ of $\bar{F},$ there is a geodesic $\sigma(t):=\bar{\sigma}(\bar{t}(t))$ of $F,$ where $\bar{t}=\bar{t}(t)$ is oriented re-parametrization, and vice-versa.
\end{definition}
Next, we recall (\cite{ChernShenRFG}, \cite{GHAM1903}) the following theorem:
\begin{theorem}{\label{thm5.1}}
	Let $F$ and $\bar{F}$ be two Finsler metrics on a manifold $M.$ Then $F$ is projectively equivalent to $\bar{F}$ if and only if 
	$$ {F}_{x^ky^\ell}y^k-{F}_{x^\ell}=0. $$ Here the spray coefficients are related by
	$G^i=\bar{G}^i+Py^i,$ where $ P=\dfrac{F_{x^k}y^k}{2F}$ is called projective factor of $F.$
\end{theorem}
It is well known that if $\bar{F}$ is  a standard Euclidean norm on $\mathbb{R}^n,$ then spray coefficients of $\bar{F}$ vanish identically, i.e., $ \bar{G}^i=0.$
As geodesics are straight lines in $\mathbb{R}^n,$ we have the following:
\begin{definition}
	For a Finsler metric $F$ on an open subset $\mathcal{U}$ of $\mathbb{R}^n,$ the geodesics of $F$ are straight lines if and only if the spray coefficients satisfy 
	$$G^i=Py^i, $$ where  $P$ is same as defined in theorem \ref{thm5.1}.
\end{definition}
Thus, for a projectively flat spherically symmetric Finsler metric on an open subset $\mathcal{U}$ of $\mathbb{R}^n,$ the coefficient  $ Q, $ occurring in the  equation (\ref{geocoeff}) vanishes identically.
\begin{definition}{\label{def5.4}}\cite{ChernShenRFG}
	A Finsler metric $F$ on an open subset $\mathcal{U}$ of $\mathbb{R}^n$ is called projectively flat if and only if all the geodesics are straight in $\mathcal{U},$ and a Finsler metric $F$ on a manifold $M$ is called locally projectively flat, if at any point, there is a local co-ordinate system $ (x^i) $ in which $F$ is projectively flat.
\end{definition}
Therefore, by theorem (\ref{thm5.1}) and definition (\ref{def5.4}), we have
\begin{theorem}
	A Finsler metric $F$ on an open subset $\mathcal{U}$ of $\mathbb{R}^n$ is projectively flat if and only if it satisfies the following system of differential equations
	$$ {F}_{x^ky^\ell}y^k-{F}_{x^\ell}=0,$$
	where $G^i=Py^i$ is spray coefficient  and $P$ is same as defined in theorem \ref{thm5.1}.
\end{theorem}

\begin{definition}  \cite{ShenRFGAIG}
	A Finsler metric $F$ on a smooth $n$-manifold $M$ is called locally dually flat if, at any point, there is a standard co-ordinate system $ (x^i, y^i) $ in $TM,$ $(x^i)$ called adapted local co-ordinate system, such that 
	$$ L_{x^ky^\ell}y^k-2L_{x^\ell}=0,\  \text{where}\ L=F^2.$$
\end{definition}

Recall \cite{XCheZSheYZho2010} the following proposition for further use.
\begin{prop}{\label{Cond_Projectively_and_Dually}}
Let $F = F(x, y)$ be a Finsler metric on an open subset  $ \mathcal{U}$ of $ \mathbb{R}^n $. $F$ is  projectively as well as  dually flat  on $ \mathcal{U}$  if and only if
$F_{x^k} = k_1FF_{y^k},$
where $k_1$ is a constant independent of $y$.
\end{prop}

\section{Spherically symmetric Finsler metrics with isotropic $E$-curvature}
There is a non-Riemannian quantity, $S$-curvature, used to measure the rate of change of the volume form of a Finsler space $ (M,F) $ along geodesics. Suppose $\gamma$ be a geodesic with $\gamma(0)=x, \ \dot{\gamma}(0)=y, $ where $x \in M \ \text{and}\  y \in T_xM,$ and let $ dv=\sigma(x) dx $ be a given volume form of $F$. Then $ S$-curvature is defined by 
$$ S(x,y)= \dfrac{d}{dt}\bigg\{\tau\bigg( \gamma(t), \dot{\gamma}(t)\bigg) \bigg\}\bigg|_{t=0},$$
where  $  \tau= \tau(x,y)$ is  the distortion function of Minkowski norm  $F_x$ on $T_xM$ given by 
$$\tau(x,y)=\ln\left(  \dfrac{\sqrt{det(g_{ij}(y))}}{\sigma(x)}\right) , \ \ y \ \in\  T_xM \backslash \{0\}. $$
Recall \cite{ChernShenRFG} that $S$-curvature can also be represented in the following way
\begin{equation*}
S(x,y)=\dfrac{\partial G^k}{\partial y^k}-y^k \dfrac{\partial}{\partial x^k}\left(\ln \sigma(x) \right). 
\end{equation*}
There is another entity associated with $S$-curvature, called $E$-curvature or mean Berwald curvature given by the $E$-tensor 
$$ \mathcal{E}:=E_{ij} \ dx^i \otimes dx^j,  $$
where
\begin{equation}{\label{Val_E-Curv}}
E_{ij}=\dfrac{1}{2} S_{y^i y^j}(x,y)
=\dfrac{1}{2} \dfrac{\partial ^2}{\partial y^i \partial y^j}\left( \dfrac{\partial G^k}{\partial y^k}-y^k \dfrac{\partial}{\partial x^k}\left(\ln \sigma(x) \right)\right)
=\dfrac{1}{2} \dfrac{\partial ^2}{\partial y^i \partial y^j}\left( \dfrac{\partial G^k}{\partial y^k}\right) 
\end{equation}
 The $E$-tensor $\mathcal{E}$ can be viewed as the family of symmetric forms 
 $$E_y: T_xM \times T_xM \longrightarrow \mathbb{R}  $$ defined by 
 $$ E_y(u,v)=E_{ij}(x,y)u^i v^j,\ \
\text{where}\ \    u=u^i \dfrac{\partial}{\partial x^i},\ v=v^i \dfrac{\partial}{\partial x^i} \in T_xM,\ x \in M. $$ 
The collection $E=\biggl\{E_y : y \in TM \backslash \{0\} \biggr\}$ is called $E$-curvature or mean Berwald curvature.\\
	During the study on spherically symmetric Finsler metrics, we find that the formula for $E$-curvature of a spherically symmetric Finsler metric given in \cite{YCheWSon2015} (Proposition $3.1$) is not correct. In the following theorem,  we give the correct version of that formula.
\begin{theorem}{\label{Thm3.1}}
Let  $F=u \psi(r,s)$ be a   spherically symmetric Finsler metric 	on an open subset  $ \mathcal{U}$ of $ \mathbb{R}^n,$ where $ r = \arrowvert x \arrowvert,\ u= \arrowvert y\arrowvert,\ v= \left\langle x, y \right\rangle,\ s=\dfrac{v}{u}$. Then $E$-curvature of $F$ is given by 

\begin{equation}{\label{E-Curvature}}
\begin{split}
E_{ij}=&\dfrac{1}{2}\biggl[\dfrac{1}{u} \biggl\{ (n+1)P_{ss} +2 Q_s +(r^2-s^2) Q_{sss} -2s Q_{ss} \biggr\} x^i x^j \\
&+\dfrac{1}{u^2}\biggl\{ -(n+1)s P_{ss} -2s Q_s -s(r^2-s^2) Q_{sss} +2s^2 Q_{ss} \biggr\}\left(x^i y^j+y^i x^j \right)\\
&+\dfrac{1}{u^3} \biggl\{ (n+1)s^2 P_{ss} +2s^2 Q_s +s^2(r^2-s^2) Q_{sss} -2s^3 Q_{ss} 
+(n+1)s P_{s} - (n+1) P \\
&+s(r^2-s^2) Q_{ss} -(r^2-s^2) Q_{s}\biggr\}y^i y^j\\
&+\dfrac{1}{u} \biggl\{ -(n+1)s P_{s} +(n+1) P -s(r^2-s^2) Q_{ss} +(r^2-s^2) Q_{s}\biggr\} \delta_{ij} \biggr]
\end{split}
\end{equation}
\end{theorem}

\begin{proof}
The geodesic spray coefficients of a spherically symmetric Finsler metric given in equation  (\ref{geocoeff}), can be re-written  as
\begin{equation*}{\label{Val_G^k}}
G^k=uP y^k+u^2Q x^k
\end{equation*}
Differentiating it partially w.r.t. $y^k$ and taking the sum over $ k=1 \ \text{to}\  n, $ we get 
\begin{equation}{\label{Val_Derv_G^k}}
\dfrac{\partial G^k}{\partial y^k}=u(n+1)P+2usQ+u(r^2-s^2)Q_s.
\end{equation}
Further, differentiating (\ref{Val_Derv_G^k}) partially w.r.t. $y^i$,  we get 
\begin{equation}{\label{Val_DDerv_G^k}}
\begin{split}
\dfrac{\partial}{\partial y^i} \left(\dfrac{\partial G^k}{\partial y^k}\right)  
=&\left\{ (n+1)P_s+2Q+(r^2-s^2) Q_{ss}\right\} x^i\\
&+ \left\{ -\dfrac{(n+1)s}{u}P_{s} +\dfrac{n+1}{u}P -\dfrac{s(r^2-s^2)}{u}Q_{ss} +\dfrac{r^2-s^2}{u} Q_{s}  \right\} y^i.
\end{split}
\end{equation}
Again, differentiating (\ref{Val_DDerv_G^k}) partially w.r.t. $y^j$, we get 
\begin{equation}{\label{Val_TDerv_G^k}}
\begin{split}
\dfrac{\partial^2}{\partial y^i \partial y^j} \left(\dfrac{\partial G^k}{\partial y^k}\right)  
&=\biggl( (n+1)P_{ss}+2Q_s+(r^2-s^2) Q_{sss}-2sQ_s\biggr) \left(\dfrac{x^i x^j}{u} -\dfrac{s}{u^2} x^i y^j \right) \\
&+\dfrac{1}{u^3}\biggl( (n+1)s P_{s} -(n+1)P +s(r^2-s^2) Q_{ss} -(r^2-s^2) Q_{s}  \biggr) y^i y^j\\
&+\dfrac{1}{u} \biggl( -(n+1)s P_{ss} -s(r^2-s^2)Q_{sss} +2s^2 Q_{ss}-2sQ_s  \biggr)\left(\dfrac{x^j y^i}{u} -\dfrac{s}{u^2} y^i y^j \right)\\
&+\dfrac{1}{u}\biggl( -(n+1)s P_{s} +(n+1)P -s(r^2-s^2) Q_{ss} +(r^2-s^2) Q_{s}  \biggr)\delta_{ij}.
\end{split}
\end{equation}
From equation \ref{Val_TDerv_G^k}, one can easily obtain equation (\ref{E-Curvature}),
which proves the theorem \ref{Thm3.1}.
\end{proof}
Using the formula \ref{E-Curvature} for $E$-curvature, we find the differential equation characterizing a spherically symmetric Finsler metric with isotropic $E$-curvature as follows:

\begin{theorem}{\label{Iso_E_Cur}}
	A spherically symmetric Finsler metrics $F=u\psi(r,s)$ 	on an open subset  $ \mathcal{U}$ of $ \mathbb{R}^n$ has isotropic $E$-curvature if and only if 
\begin{equation}
(n+1)(P-sP_s)+(r^2-s^2)(Q_s-sQ_{ss})=(n+1)c(x)(\psi-s\psi_s),
\end{equation}
where $c=c(x)$ is a scalar function on $M$.
\end{theorem}
The equation \ref{Iso_E_Cur} is same as the one given in  \cite{YCheWSon2015}.
\section{Rigidity of  spherically symmetric Finsler metrics}
The main aim of this section is to prove the rigidity theorem \ref{Main_Result}, but before that we prove the following theorem. 
\begin{theorem}{\label{SSFM_Cond_Projectively_and_Dually}}
Let  $F=u \psi(r,s)$ be a spherically symmetric Finsler metric 	on an open subset  $ \mathcal{U}$ of $ \mathbb{R}^n,$ where $ r = \arrowvert x \arrowvert,\ u= \arrowvert y\arrowvert,\ v= \left\langle x, y \right\rangle,\ s=\dfrac{v}{u}$. Then $F$ is  projectively as well as  dually flat  on $ \mathcal{U}$  if and only if following equations are satisfied:
	$$ \dfrac{1}{r} \psi_r=k_1 \psi \psi_s,\ \ 
	\psi_s= k_1\left( \psi^2-s \psi \psi_s \right),  $$
	where $k_1$ is a constant independent of $y$.
\end{theorem}
\begin{proof}
	We have 
	$$F=u \psi(r,s)$$
	Differentiating it partially w.r.t. $ x^k $ and $y^k$ respectively, we get
	$$ F_{x^k}=\dfrac{u}{r} \psi_r x^k+\psi_s y^k,$$
	and 
	$$ F_{y^k}= \psi_s x^k- \dfrac{s}{u} \psi_s y^k+ \dfrac{1}{u} \psi y^k. $$
	By proposition \ref{Cond_Projectively_and_Dually}, $F$ is projectively as well as   dually flat  on $ \mathcal{U}$  if and only if
	$$F_{x^k} = k_1FF_{y^k},$$
	where $k_1$ is a constant independent of $y$.\\
	Therefore, $F$ is projectively as well as   dually flat  on $ \mathcal{U}$  if and only if
	
	$$ \dfrac{u}{r} \psi_r x^k+\psi_s y^k=k_1 u \psi\left( \psi_s x^k- \dfrac{s}{u} \psi_s y^k+ \dfrac{1}{u} \psi y^k \right).$$
From this we conclude that $F$ is projectively as well as   dually flat  on $ \mathcal{U}$  if and only if 
	$$ \dfrac{1}{r} \psi_r=k_1 \psi \psi_s,\ \ 
\psi_s= k_1\left( \psi^2-s \psi \psi_s \right).  $$
\end{proof}


\begin{cor}{\label{corr}}
	For a  spherically symmetric Finsler metric  $F=u \psi(r,s)$ on an open subset  $ \mathcal{U}$ of $ \mathbb{R}^n,$ where $ r = \arrowvert x \arrowvert,\ u= \arrowvert y\arrowvert,\ v= \left\langle x, y \right\rangle,\ s=\dfrac{v}{u}$,  $P$ and $Q$ in its geodesic spray coefficients satisfy the following:
	\begin{equation}{\label{PFDF_Geo_coeff_P_Q}}
	P=\dfrac{k_1}{2} \psi, \ \ Q=0.
	\end{equation} 
\end{cor}
\vspace*{.1cm}
\noindent \textbf{\Large Proof of rigidity theorem \ref{Main_Result} }\\
\vspace*{.05cm}\\
Since  $Q=0$ for projectively flat spherically symmetric Finsler metric,  by the theorem \ref{Iso_E_Cur}, we see that  
$F$ has isotropic $E$-curvature if and only if 

\begin{equation}{\label{Der_P_Psi}}
P-sP_s=c(x)(\psi-s\psi_s),
\end{equation}
where $c=c(x)$ is a scalar function on $M$.\\
From (\ref{Der_P_Psi}), we get
$$\dfrac{d}{ds}\left(\dfrac{P}{s} \right) = c(x) \dfrac{d}{ds}\left(\dfrac{\psi}{s} \right) $$
which implies
$$  P = c(x) \psi -s \gamma(r). $$
By corollary \ref{corr}, we have $P=\dfrac{k_1}{2} \psi.$\\
Therefore, 
$$ \psi= \dfrac{2s}{2c(x)-k_1} \gamma(r),  $$
and consequently 
\begin{equation}
P=k(r) s, \ \ \text{where  \ } k(r)= \dfrac{k_1\ \gamma(r)}{2c(x)-k_1}. 
\end{equation}
Putting the values of $P$ and $Q$ in equation (\ref{valP}), we get 
\begin{equation}{\label{Hash}}
2rsk(r) \psi -s \psi_r-r \psi_s=0.
\end{equation}
Differentiating (\ref{Hash}) w.r.t. $s$, we get 
\begin{equation}{\label{Star}}
2rk(r) \psi +2rsk(r) \psi_s-r\psi_{ss}-s\psi_{rs}-\psi_r=0.
\end{equation}
From equation (\ref{valQ}), we have 
\begin{equation}{\label{2Star}}
r \psi_{ss}+s \psi_{rs}-\psi_r =0.
\end{equation}
Solving equations (\ref{Star}) and (\ref{2Star}), we get
\begin{equation}{\label{2Hash}}
rk(r) \psi +rsk(r) \psi_s-\psi_r=0.
\end{equation}
By the equations (\ref{Hash}) and (\ref{2Hash}), we get

$$ \psi= k_2(r) \sqrt{1+s^2 k(r)}, \ \ \text{where}\ \ k_2(r)\  \text{is a function independent of s}, $$
which yields
$$ F= k_2(r) \sqrt{\arrowvert y\arrowvert^2+k(r) \left\langle x, y \right\rangle^2} $$
which is a Riemannian metric.

\end{document}